\documentclass[a4paper,11pt]{amsart}
\usepackage{latexsym}
\usepackage{amsmath,amssymb}
\usepackage{amsmath}

\usepackage[all]{xy}
\usepackage{enumerate}

\newtheorem{theo}{Theorem}[section]
\newtheorem{coll}[theo]{Corollary}
\newtheorem{lemm}[theo]{Lemma}
\newtheorem{prop}[theo]{Proposition}
\newtheorem{defn}[theo]{Definition}
\newtheorem{ex}[theo]{Example}
\newtheorem{rem}[theo]{Remark}

\newcommand{\Hom}{{\rm Hom}}

\renewcommand{\hom}[3]{\mathrm{Hom}_{#1}(#2,#3)}
\newcommand{\HOM}[3]{\mathrm{HOM}_{#1}(#2,#3)}

\begin{document}
\sloppy

\title[Strongly Rickart objects]{Strongly Rickart objects in abelian categories}

\author[S.Crivei]{Septimiu Crivei}

\address{Faculty of Mathematics and Computer Science, Babe\c s-Bolyai University, Str. M. Kog\u alniceanu 1,
400084 Cluj-Napoca, Romania} \email{crivei@math.ubbcluj.ro}

\author[G. Olteanu]{Gabriela Olteanu}

\address{Department of Statistics-Forecasts-Mathematics, Babe\c s-Bolyai University, Str. T. Mihali 58-60, 400591
Cluj-Napoca, Romania} \email{gabriela.olteanu@econ.ubbcluj.ro}


\begin{abstract} We introduce and study (dual) strongly relative Rickart objects in abelian categories. 
We prove general properties, we analyze the behaviour with respect to (co)products, and we study the transfer via functors.
We also give applications to Grothendieck categories, (graded) module categories and comodule categories.
Our theory of (dual) strongly relative Rickart objects may be employed in order to 
study strongly relative regular objects and (dual) strongly relative Baer objects in abelian categories.
\end{abstract}

\date{February 5, 2018}

\maketitle

\section{Introduction}

Among the most significant concepts in ring theory there was the following one: 
a ring $R$ is called \emph{von Neumann regular} if for every $a\in R$ there exists $b\in R$ such that $a=aba$ \cite{vN}.
The development of module theory has led to its module-theoretic generalization due to Zelmanowitz: 
a right $R$-module $M$ is called \emph{(Zelmanowitz) regular} if for every $f\in \Hom_R(R,M)\cong M$ there exists 
$g\in \Hom_R(M,R)$ such that $f=fgf$ \cite{Zel}. A natural further step was its categorical version, 
introduced by D\u asc\u alescu, N\u ast\u asescu, Tudorache and D\u au\c s as follows: 
given two objects $M$ and $N$ of an arbitrary category $\mathcal{C}$, $N$ is called \emph{$M$-regular} 
if for every $f\in \Hom_{\mathcal{C}}(M,N)$, there exists a morphism $g\in \Hom_{\mathcal{C}}(N,M)$ such that 
$f=fgf$. Relative regular objects were studied in the series of papers \cite{CK,CO,DNT,DNTD,Daus}. 
Regular morphisms and relative regular modules were also studied by Kasch and Mader \cite{KM}, 
and Lee, Rizvi and Roman \cite{LRR13}.

Given two objects $M$ and $N$ in an abelian category $\mathcal{A}$, 
$N$ is $M$-regular if and only if for every morphism $f:M\to N$ in $\mathcal{A}$, 
its kernel ${\rm Ker}(f)$ is a direct summand of $M$ and its image ${\rm Im}(f)$ is a direct summand of $N$ 
\cite[Proposition~3.1]{DNTD}. These two conditions characterizing relative regularity in abelian categories 
can be considered separately. This was done by Crivei and K\"or \cite{CK}, and Crivei and Olteanu \cite{CO}, 
under the names of \emph{relative Rickart} and \emph{dual relative Rickart} objects. In module categories, 
Rickart and dual Rickart objects were mainly studied by Lee, Rizvi and Roman \cite{LRR10,LRR11,LRR12}, 
as generalizations of Baer and dual Baer modules, studied by Rizvi and Roman \cite{RR04,RR09}, 
Keskin T\"ut\"unc\"u, Smith and Toksoy \cite{KST}, and Keskin T\"ut\"unc\"u and Tribak \cite{KT}. 
Unlike the previous papers, the framework of abelian categories allows one to study just one of the two concepts, 
and to obtain the corresponding results for the other one simply by using the duality principle. 
Relative Rickart and dual relative Rickart objects in abelian categories have been shown to be the suitable tools 
for a unified treatment of relative regular objects, Baer objects and dual Baer objects in abelian categories. 

The purpose of the present paper is to develop a similar theory starting from an important subclass of 
von Neumann regular rings, namely strongly regular rings. A ring $R$ is called \emph{strongly regular} if 
for every $a\in R$ there exists $b\in R$ such that $a=a^2b$ \cite{AK}. We introduce the following concepts.
Given two objects $M$ and $N$ of an abelian category $\mathcal{A}$, $N$ will be called:  
(1) \emph{(dual) strongly $M$-Rickart} if the kernel (the image) of every morphism $f:M\to N$ 
is a fully invariant direct summand of $M$ (of $N$);
(2) \emph{(dual) strongly self-Rickart} if $N$ is (dual) strongly $N$-Rickart; 
(3) \emph{strongly $M$-regular} if $N$ is strongly $M$-Rickart and dual strongly $M$-Rickart;
(4) \emph{strongly self-regular} if $N$ is strongly $N$-regular;
(5) \emph{(dual) strongly $M$-Baer} if for every family $(f_i)_{i\in I}$ with each $f_i\in \Hom_{\mathcal{A}}(M,N)$, $\bigcap_{i\in I}
{\rm Ker}(f_i)$ ($\sum_{i\in I} {\rm Im}(f_i)$) is a fully invariant direct summand of $M$ (of $N$); 
equivalently, when the required (co)product does exist, 
$N^I$ is strongly $M$-Rickart ($N$ is dual strongly $M^{(I)}$-Rickart) for every set $I$;
(6) \emph{(dual) strongly self-Baer} if $N$ is (dual) strongly $N$-Baer.
Note that strongly Rickart and dual strongly Rickart modules have been recently studied in \cite{AI14,AI15,Atani,WangY}.

Now it should be clear that properties on strongly relative Rickart objects in abelian categories can be used 
to obtain corresponding properties on all the other related concepts, namely 
dual strongly relative Rickart objects (by the duality principle), 
strongly relative regular objects (which are strongly relative Rickart and dual strongly relative Rickart),
strongly relative Baer objects (as particular strongly relative Rickart objects) and 
dual strongly relative Baer objects (by the duality principle).

We shall show in the companion paper \cite{CO2} how our theory of (dual) strongly relative Rickart objects 
may be effectively employed in order to study strongly relative regular objects and 
(dual) strongly relative Baer objects in abelian categories. 

Coming back to the contents of the present paper, we should mention that 
the statements of our results usually have two parts, out of which we only prove the first one, 
the second one following by the duality principle in abelian categories.

In Section 2 we introduce strongly relative Rickart objects in abelian categories. We show that 
an object $M$ of an abelian category $\mathcal{A}$ is strongly self-Rickart if and only if 
$M$ is self-Rickart and weak duo if and only if $M$ is self-Rickart and ${\rm End}_{\mathcal{A}}(M)$ is abelian.
If $r:M\to M'$ is an epimorphism and $s:N'\to N$ is a monomorphism in an abelian category, then we prove that 
if $N$ is strongly $M$-Rickart, then $N'$ is strongly $M'$-Rickart. 
For a short exact sequence $0\to N_1\to N\to N_2\to 0$ and an object $M$ of an abelian category
such that $N_1$ and $N_2$ are strongly $M$-Rickart, we show that $N$ is strongly $M$-Rickart.  

In Section 3 we study (co)products of strongly relative Rickart objects. We prove that if 
$M$, $N_1,\dots,N_n$ are objects of an abelian category $\mathcal{A}$, 
then $\bigoplus_{i=1}^n N_i$ is strongly $M$-Rickart if and only if $N_i$ is strongly $M$-Rickart for every $i\in \{1,\dots,n\}$.
If $M$ is a weak duo object with the strong summand intersection property of an abelian category $\mathcal{A}$, 
and $(N_i)_{i\in I}$ is a family of objects of $\mathcal{A}$ having a product, then we show that 
$\prod_{i\in I} N_i$ is strongly $M$-Rickart if and only if $N_i$ is strongly $M$-Rickart for every $i\in I$.
We also show that if $M=\bigoplus_{i\in I}M_i$ is a direct sum decomposition of 
an object $M$ of an abelian category $\mathcal{A}$, then $M$ is strongly self-Rickart if and only if 
$M_i$ is strongly self-Rickart for every $i\in I$, and ${\rm Hom}_{\mathcal{A}}(M_i,M_j)=0$ for every $i,j\in I$ with $i\neq j$.
We deduce a corollary on the structure of strongly self-Rickart modules over a Dedekind domain. 

In Section 4 we discuss the transfer of the strong relative Rickart property via functors. 
We show various results involving fully faithful functors, adjoint triples and adjoint pairs of functors. 
Thus, if $F:\mathcal{A}\to \mathcal{B}$ is a left exact fully faithful covariant functor between abelian
categories, and $M$ and $N$ are objects of $\mathcal{A}$, then $N$ is strongly $M$-Rickart in $\mathcal{A}$ 
if and only if $F(N)$ is strongly $F(M)$-Rickart in $\mathcal{B}$. We give applications to adjoint triple of functors 
as well as to reflective and coreflective abelian subcategories of an abelian category. Next we consider 
an adjoint pair $(L,R)$ of covariant functors $L:\mathcal{A}\to \mathcal{B}$ and
$R:\mathcal{B}\to \mathcal{A}$ between abelian categories with counit of adjunction $\varepsilon:LR\to 1_{\mathcal{B}}$. 
If $L$ is exact, and $M$, $N$ are objects of $\mathcal{B}$ such that $M,N\in {\rm Stat}(R)$ (i.e., $\varepsilon_M$ and 
$\varepsilon_N$ are isomorphisms), then we prove that the following are equivalent: 
$(i)$ $N$ is strongly $M$-regular in $\mathcal{B}$; 
$(ii)$ $R(N)$ is strongly $R(M)$-regular in $\mathcal{A}$ and for every morphism $f:M\to N$, ${\rm Ker}(f)$ is $M$-cyclic; 
$(iii)$ $R(N)$ is strongly $R(M)$-regular in $\mathcal{A}$ and for every morphism $f:M\to N$, ${\rm Ker}(f)\in {\rm Stat}(R)$.
Finally, we derive a theorem relating the strong self-Rickart properties of a module and of its
endomorphism ring. More precisely, the following are shown to be equivalent for a right $R$-module $M$ with $S={\rm End}_R(M)$: 
$(i)$ $M$ is a strongly self-Rickart right $R$-module;  
$(ii)$ $S$ is a strongly self-Rickart right $S$-module and for every $f\in S$, ${\rm Ker}(f)$ is $M$-cyclic;
$(iii)$ $S$ is a strongly self-Rickart right $S$-module and for every $f\in S$, ${\rm Ker}(f)\in {\rm Stat}({\rm Hom}_R(M,-))$;
$(iv)$ $S$ is a strongly self-Rickart right $S$-module and for every $f\in S$, ${\rm ker}(f)$ is a locally split monomorphism;
$(v)$ $S$ is a strongly self-Rickart right $S$-module and $M$ is $k$-quasi-retractable.

Throughout the paper we illustrate our results in Grothendieck categories, (graded) module categories and comodule categories.

\section{(Dual) strongly relative Rickart objects}

Let $\mathcal{A}$ be an abelian category. For every morphism $f:M\to N$ in $\mathcal{A}$ we have the following 
commutative diagram involving its kernel, cokernel, image and coimage:
$$\SelectTips{cm}{}
\xymatrix{
{\rm Ker}(f) \ar[r]^-{{\rm ker}(f)} & M \ar[r]^f \ar[d]_{{\rm coim}(f)} & N \ar[r]^-{{\rm coker}(f)} & {\rm Coker}(f) \\
& {\rm Coim}(f) \ar[r]_-{\overline{f}} & {\rm Im}(f) \ar[u]_{{\rm im}(f)} &  
}$$
where the row is exact and $\overline{f}$ is an isomorphism. 

Recall that a morphism $f:M\to N$ is called a \emph{section} if there is a morphism
$f':N\to M$ such that $f'f=1_M$, and a \emph{retraction} if there is a morphism $f':N\to
M$ such that $ff'=1_N$.  

Let us also recall the following definition.

\begin{defn} \cite[Definition~2.2]{CK} \rm Let $M$ and $N$ be objects of an abelian category $\mathcal{A}$. Then $N$ is called:
\begin{enumerate}
\item \emph{$M$-Rickart} if the kernel of every morphism $f:M\to N$ is a section, 
or equivalently, the coimage of every morphism $f:M\to N$ is a retraction.
\item \emph{dual $M$-Rickart} if the cokernel of every morphism $f:M\to N$ is a retraction, 
or equivalently, the image of every morphism $f:M\to N$ is a section.
\item \emph{self-Rickart} if $N$ is $N$-Rickart.
\item \emph{dual self-Rickart} if $N$ is dual $N$-Rickart.
\end{enumerate}
\end{defn}

In order to define the main concepts of our paper, which is a specialization of the above definition, 
we need to introduce the following generalization of fully invariant submodule of a module.

\begin{defn} \rm Let $\mathcal{A}$ be an abelian category. 
\begin{enumerate}
\item A kernel (or monomorphism) $k:K\to M$ in $\mathcal{A}$ is called \emph{fully invariant} if for every morphism $h:M\to M$, 
$hk$ factors through $k$. A subobject $K$ of $M$ is called \emph{fully invariant} 
if the inclusion monomorphism $k:K\to M$ is fully invariant.
\item A cokernel (or epimorphism) $c:M\to C$ in $\mathcal{A}$ is called \emph{fully coinvariant} if for every morphism $h:M\to M$, 
$ch$ factors through $c$. A factor object $C$ of $M$ is called \emph{fully coinvariant} 
if the natural epimorphism $c:M\to C$ is fully coinvariant.
\end{enumerate}
\end{defn}

\begin{ex} \rm In a module category, an inclusion monomorphism $k:K\to M$ is fully invariant if and only if 
for every homomorphism $h:M\to M$, $hk=k\alpha$ for some homomorphism $\alpha:K\to K$ if and only if
for every homomorphism $h:M\to M$, $h(K)={\rm Im}(hk)\subseteq {\rm Im}(k)=K$ if and only if 
$K$ is a fully invariant submodule of $M$.
\end{ex}

The following lemma is immediate.

\begin{lemm} \label{l1:comp} Let $\mathcal{A}$ be an abelian category. 
\begin{enumerate}
\item The composition of two fully invariant kernels is a fully invariant kernel.
\item The composition of two fully coinvariant cokernels is a fully coinvariant cokernel.
\end{enumerate}
\end{lemm}

\begin{lemm} \label{l1:eq} Let $\mathcal{A}$ be an abelian category. Then a kernel $k:K\to M$ in $\mathcal{A}$ is fully invariant 
if and only if its cokernel $c:M\to C$ is fully coinvariant. 
\end{lemm}

\begin{proof} We only need to show one implication, the other one following by duality. 
Assume that the cokernel $c:M\to C$ is fully coinvariant. Let $h:M\to M$ be a morphism. 
Then there exists a morphism $\gamma:C\to C$ such that $ch=\gamma c$. 
Since $chk=\gamma ck=0$, there exists a morphism $\alpha:K\to K$ such that $hk=k\alpha$. 
Hence the kernel $k:K\to M$ is fully invariant.
\end{proof}

We also need the following definition, which is a generalization of the corresponding one for modules. 

\begin{defn} \rm An object $M$ of an abelian category $\mathcal{A}$ is called \emph{weak duo} 
if every section $K\to M$ is fully invariant, or equivalently, every retraction $M\to C$ is fully coinvariant.
\end{defn}

Now we introduce the main concepts of the paper.

\begin{defn} \rm Let $M$ and $N$ be objects of an abelian category $\mathcal{A}$. Then $N$ is called:
\begin{enumerate}
\item \emph{strongly $M$-Rickart} if the kernel of every morphism $f:M\to N$ is a fully invariant section, 
or equivalently, the coimage of every morphism $f:M\to N$ is a fully coinvariant retraction.
\item \emph{dual strongly $M$-Rickart} if the cokernel of every morphism $f:M\to N$ is a fully coinvariant retraction, 
or equivalently, the image of every morphism $f:M\to N$ is a fully invariant section.
\item \emph{strongly self-Rickart} if $N$ is strongly $N$-Rickart.
\item \emph{dual strongly self-Rickart} if $N$ is dual strongly $N$-Rickart.
\end{enumerate}
\end{defn}

The following lemma is immediate.

\begin{lemm} \label{l1:split} Let $M$ and $N$ be objects of an abelian category $\mathcal{A}$.
\begin{enumerate}
\item If $N$ is strongly $M$-Rickart, then every cokernel $f:M\to N$ is a fully coinvariant retraction.
\item If $N$ is dual strongly $M$-Rickart, then every kernel $f:M\to N$ is a fully invariant section.
\end{enumerate}
\end{lemm}

\begin{prop} \label{p1:wduo} Let $M$ and $N$ be objects of an abelian category $\mathcal{A}$. 
\begin{enumerate}
\item Assume that every direct summand of $M$ is isomorphic to a subobject of $N$. 
Then $N$ is strongly $M$-Rickart if and only if $N$ is $M$-Rickart and $M$ is weak duo.
\item Assume that every direct summand of $N$ is isomorphic to a factor object of $M$. 
Then $N$ is dual strongly $M$-Rickart if and only if $N$ is dual $M$-Rickart and $N$ is weak duo.
\end{enumerate}
\end{prop}

\begin{proof} (1) Suppose first that $N$ is strongly $M$-Rickart. Clearly, $N$ is $M$-Rickart. 
Let $k:K\to M$ be a section. Write $M=K\oplus J$ for some subobject $J$ of $M$. 
There is an isomorphism $\varphi:J\to L$ for some subobject $L$ of $N$. Denote by 
$l:L\to N$ the inclusion monomorphism and by $p:M\to J$ the canonical projection. Consider the morphism 
$l\varphi p:M\to N$. Since $N$ is strongly $M$-Rickart, 
$k={\rm ker}(p)={\rm ker}(l\varphi p)$ is a fully invariant section. Hence $M$ is weak duo.

The converse is clear. 
\end{proof}

\begin{coll} \label{c1:wduo} Let $M$ be an object of an abelian category $\mathcal{A}$. Then:
\begin{enumerate} 
\item $M$ is strongly self-Rickart if and only if $M$ is self-Rickart and weak duo. 
\item $M$ is dual strongly self-Rickart if and only if $M$ is dual self-Rickart and weak duo.
\end{enumerate}
\end{coll}

\begin{coll} \label{c1:indec} Let $M$ be an indecomposable object of an abelian category $\mathcal{A}$. Then:
\begin{enumerate} 
\item $M$ is strongly self-Rickart if and only if $M$ is self-Rickart. 
\item $M$ is dual strongly self-Rickart if and only if $M$ is dual self-Rickart.
\end{enumerate}
\end{coll}

\begin{ex} \rm By Corollary \ref{c1:indec} and \cite[Example~2.5]{LRR11}, 
the indecomposable abelian group $\mathbb{Z}$ is strongly self-Rickart, 
but not dual strongly self-Rickart, while the indecomposable abelian group  
$\mathbb{Z}_{p^{\infty}}$ (the Pr\"ufer group for some prime $p$) is 
dual strongly self-Rickart, but not strongly self-Rickart. 
\end{ex}

Now let us recall some ring-theoretic concepts. A ring $R$ is called 
\emph{abelian} if every idempotent element of $R$ is central. An element $a$ of a ring $R$ is called 
\emph{left (right) semicentral} if $ba=aba$ ($ab=aba$) for every $b\in R$. 
Note that a ring $R$ is abelian if and only if every idempotent element of $R$ is left semicentral 
if and only if every idempotent element of $R$ is right semicentral (e.g., see \cite[p.~17]{Wei}). 
A short argument of the non-trivial part of these equivalences is as follows: 
if every idempotent element of $R$ is left semicentral,
then for every idempotent $e\in R$ and for every element $b\in R$, $e$ and $1-e$ are left semicentral, 
which implies that $be=ebe=(1-e)b(1-e)-b(1-e)+eb=b(1-e)-b(1-e)+eb=eb$, hence $e$ is central.

The following lemma generalizes \cite[Lemma~1.9~(ii)]{BMR}. 
It uses the property that every abelian category $\mathcal{A}$ has split idempotents, that is, 
for every idempotent $e:M\to M$, there exist an object $K$  
and morphisms $k:K\to M$ and $p:M\to K$ such that $kp=e$ and $pk=1_K$.

\begin{lemm} \label{l1:semic} Let $M$ be an object of an abelian category $\mathcal{A}$, and $e=e^2\in {\rm End}_{\mathcal{A}}(M)$. 
Consider an object $K$ of $\mathcal{A}$ and morphisms $k:K\to M$ and $p:M\to K$ such that 
$kp=e$ and $pk=1_K$. Then:
\begin{enumerate}
\item The kernel $k$ is fully invariant if and only if $e$ is left semicentral.  
\item The cokernel $p$ is fully coinvariant if and only if $e$ is right semicentral.
\end{enumerate}
\end{lemm}

\begin{proof} (1) Assume that $k$ is fully invariant. Let $h\in {\rm End}_{\mathcal{A}}(M)$. 
Then $hek=k\alpha$ for some morphism $\alpha:K\to K$. 
It follows that $ehek=ek\alpha=kpk\alpha=k\alpha=hek$, hence $ehe=ehekp=hekp=he$. Thus, $e$ is right semicentral. 

Conversely, assume that $e$ is right semicentral. Let $h\in {\rm End}_{\mathcal{A}}(M)$.
Then $hk=hkpk=hek=ehek=kphkpk=kphk$, and so $k$ is fully invariant.
\end{proof}

\begin{prop} \label{p1:strring} Let $M$ be an object of an abelian category $\mathcal{A}$. Then:
\begin{enumerate}
\item $M$ is strongly self-Rickart if and only if $M$ is self-Rickart and ${\rm End}_{\mathcal{A}}(M)$ is abelian.
\item $M$ is dual strongly self-Rickart if and only if $M$ is dual self-Rickart and ${\rm End}_{\mathcal{A}}(M)$ is abelian.
\end{enumerate}
\end{prop}

\begin{proof} (1) Assume first that $M$ is strongly self-Rickart. Then $M$ is self-Rickart.
Let $e=e^2\in {\rm End}_{\mathcal{A}}(M)$. Since every idempotent splits, 
there exist an object $K$ and morphisms $k:K\to M$ and $p:M\to K$ such that $kp=e$ and $pk=1_K$. 
Since $M$ is strongly self-Rickart, the kernel $k$ is fully invariant. Then $e$ is left semicentral by Lemma~\ref{l1:semic}. 
Hence every idempotent in ${\rm End}_{\mathcal{A}}(M)$ is left semicentral, 
which shows that ${\rm End}_{\mathcal{A}}(M)$ is abelian.

Conversely, assume that $M$ is self-Rickart and ${\rm End}_{\mathcal{A}}(M)$ is abelian.
We claim that $M$ is weak duo. To this end, let $k:K\to M$ be a section and $p:M\to K$ the canonical projection. 
Then $pk=1_K$ and $e=kp\in {\rm End}_{\mathcal{A}}(M)$ is idempotent. Since ${\rm End}_{\mathcal{A}}(M)$ is abelian, 
$e$ is left semicentral, hence the kernel $k$ is fully invariant by Lemma~\ref{l1:semic}. Then $M$ is weak duo. 
Finally, $M$ is strongly self-Rickart by Corollary~\ref{c1:wduo}.  
\end{proof}

\begin{rem} \rm By Proposition \ref{p1:strring} it follows that a right $R$-module $M$ is strongly self-Rickart 
if and only if the kernel of every endomorphism of $M$ is generated by a right semicentral idempotent, 
while a right $R$-module $M$ is dual strongly self-Rickart if and only if the image of every endomorphism of $M$ 
is generated by a right semicentral idempotent.
\end{rem}

\begin{ex} \rm Let $K$ be a field. Then $R=\begin{pmatrix}K&K\\0&K\end{pmatrix}$ 
is a self-Rickart right $R$-module, which is not strongly self-Rickart \cite[Example~2.16]{Atani}. 
Also, the full $2\times 2$ matrix ring $S=M_2(K)$ over $K$ is a dual self-Rickart right $S$-module, 
which is not dual strongly self-Rickart \cite[Example~3.14]{WangY}. 
\end{ex}

The following result will be very useful. 

\begin{theo} \label{t1:epimono} Let $r:M\to M'$ be an epimorphism and $s:N'\to N$ a monomorphism in an abelian category
$\mathcal{A}$. 
\begin{enumerate} 
\item If $N$ is strongly $M$-Rickart, then $N'$ is strongly $M'$-Rickart.
\item If $N$ is dual strongly $M$-Rickart, then $N'$ is dual strongly $M'$-Rickart.
\end{enumerate}
\end{theo}

\begin{proof} (1) We use the idea of the proof of \cite[Theorem~2.10]{CK}. 

Let $f:M'\to N'$ be a morphism in $\mathcal{A}$. Consider the morphism $sfr:M\to N$. Since $N$ is strongly 
$M$-Rickart, ${\rm ker}(sfr)$ is a fully invariant section. But $s$ is a monomorphism, 
hence ${\rm ker}(sfr)={\rm ker}(fr)$. We have the following induced commutative diagram:
$$\SelectTips{cm}{}
\xymatrix{
 & 0 \ar[d] & 0 \ar[d] & \\
 & X \ar@{=}[r] \ar[d]_i & X \ar[d]^{li} & \\ 
0 \ar[r] & L \ar[r]^-{l} \ar[d]_p & M \ar[r]^-{fr} \ar[d]^r & N' \ar@{=}[d] \ar[r]^s & N \ar@{=}[d] \\
0 \ar[r] & K \ar[r]_-{k} \ar[d] & M' \ar[r]_f \ar[d] & N' \ar[r]_s & N \\
 & 0 & 0 &
}$$
with left exact rows and exact columns, in which the square $LMKM'$ is a pushout by \cite[Lemma~2.9]{CK}. 
Since pushouts preserve sections, it follows that ${\rm ker}(f)=k$ is a section. 

We know that $l={\rm ker}(fr)$ is a fully invariant section. We claim that $k:K\to M'$ is also a fully invariant section. 
To this end, let $h:M'\to M'$ be a morphism. The pushout $LMKM'$ induces the following commutative diagram with split exact rows:
$$\SelectTips{cm}{}
\xymatrix{
0 \ar[r] & L \ar[r]^-{l} \ar[d]_p & M \ar[r]^-{v} \ar[d]^r & L' \ar@{=}[d] \ar[r] & 0 \\
0 \ar[r] & K \ar[r]_-{k} & M' \ar[r]_u & L' \ar[r] & 0 \\
}$$
There exists a monomorphism $w:L'\to M$ such that $vw=1_{L'}$. Consider the morphism $wuhrl:L\to M$. 
Since $l$ is a fully invariant section, there exists a morphism $\alpha:L\to L$ such that $wuhrl=l\alpha$. 
Then $uhkp=uhrl=vwuhrl=0$, whence we have $uhk=0$, because $p$ is an epimorphism. It follows that there exists a morphism 
$\gamma:L'\to L'$ such that $uh=\gamma u$. Hence $u$ is a fully coinvariant cokernel. Then $k$ is a fully invariant kernel by 
Lemma \ref{l1:eq}. This shows that $N'$ is strongly $M'$-Rickart.
\end{proof}

\begin{coll} \label{c1:summand} Let $M$ and $N$ be objects of an abelian category $\mathcal{A}$, $M'$ a direct summand 
of $M$ and $N'$ a direct summand of $N$. 

(1) If $N$ is strongly $M$-Rickart, then $N'$ is strongly $M'$-Rickart.

(2) If $N$ is dual strongly $M$-Rickart, then $N'$ is dual strongly $M'$-Rickart.
\end{coll}

\begin{theo} \label{t1:extensions} Let $\mathcal{A}$ be an abelian category. 
\begin{enumerate} \item Consider a short exact sequence $$0\to N_1\to N\to N_2\to 0$$ and an object $M$ of $\mathcal{A}$
such that $N_1$ and $N_2$ are strongly $M$-Rickart. Then $N$ is strongly $M$-Rickart.  
\item Consider a short exact sequence $$0\to M_1\to M\to M_2\to 0$$ and an object $N$ of $\mathcal{A}$ such that $N$ is
dual strongly $M_1$-Rickart and dual strongly $M_2$-Rickart. Then $N$ is dual strongly $M$-Rickart. 
\end{enumerate}
\end{theo}

\begin{proof} (1) We use the idea of the proof of \cite[Theorem~2.12]{CK}. 

Let $f:M\to N$ be a morphism in $\mathcal{A}$. We have the following induced commutative diagram:
$$\SelectTips{cm}{}
\xymatrix{
 & 0 \ar[d] & 0 \ar[d] & \\
 & K_1 \ar@{=}[r] \ar[d]_{i_1} & K_1 \ar[d]^{i_2} & \\ 
0 \ar[r] & K_2 \ar[r]^-{k_2} \ar[d]_{f_1} & M \ar[r]^-{f_2} \ar[d]^f & N_2 \ar@{=}[d] \ar[r] & 0 \\
0 \ar[r] & N_1 \ar[r] & N \ar[r] & N_2 \ar[r] & 0 
}$$
with exact rows and columns. Since $N_2$ is strongly $M$-Rickart, $k_2$ is a fully invariant section. 
Then there is an epimorphism $p_2:M\to K_2$ such that $p_2k_2=1_{K_2}$. 
Since $N_1$ is strongly $M$-Rickart, it follows that $N_1$ is also strongly $K_2$-Rickart by Theorem
\ref{t1:epimono}. Then $i_1$ is a fully invariant section. It follows that 
$i_2=k_2i_1$ is a fully invariant section by Lemma \ref{l1:comp}. This shows that $N$ is strongly $M$-Rickart.
\end{proof}

\begin{ex} \rm Let $\mathbb{Z}_4=\mathbb{Z}/4\mathbb{Z}$ and consider the natural epimorphism $\mathbb{Z}\to
\mathbb{Z}_4$ of abelian groups. Since $\Hom_{\mathbb{Z}}(\mathbb{Z}_4,\mathbb{Z})=0$, $\mathbb{Z}$ is clearly a
strongly $\mathbb{Z}_4$-Rickart abelian group. But $\mathbb{Z}_4$ is not a $\mathbb{Z}_4$-Rickart abelian group
\cite[Example~2.5]{LRR12}, and so $\mathbb{Z}_4$ is not strongly $\mathbb{Z}_4$-Rickart.
\end{ex}

\section{(Co)products of (dual) strongly relative Rickart objects}

The main theorem on finite (co)products involving (dual) strongly relative Rickart objects is the following one, which is now
easily deduced from our previous results.  

\begin{theo} \label{t1:ds} Let $\mathcal{A}$ be an abelian category.
\begin{enumerate} 
\item Let $M$ and $N_1,\dots,N_n$ be objects of $\mathcal{A}$. Then $\bigoplus_{i=1}^n N_i$ is strongly $M$-Rickart if and only
if $N_i$ is strongly $M$-Rickart for every $i\in \{1,\dots,n\}$.
\item Let $M_1,\dots,M_n$ and $N$ be objects of $\mathcal{A}$. Then $N$ is dual strongly $\bigoplus_{i=1}^n M_i$-Rickart if and
only if $N$ is dual strongly $M_i$-Rickart for every $i\in \{1,\dots,n\}$.
\end{enumerate}
\end{theo}

\begin{proof} This follows by Corollary~\ref{c1:summand} and Theorem~\ref{t1:extensions}.
\end{proof}

We also have the following result on arbitrary (co)products under some finiteness conditions.

\begin{coll} \label{c1:fg} Let $\mathcal{A}$ be an abelian category.
\begin{enumerate} \item Assume that $\mathcal{A}$ has coproducts, let $M$ be a finitely generated object of
$\mathcal{A}$, and let $(N_i)_{i\in I}$ be a family of objects of $\mathcal{A}$. Then $\bigoplus_{i\in I}
N_i$ is strongly $M$-Rickart if and only if $N_i$ is strongly $M$-Rickart for every $i\in I$.

\item Assume that $\mathcal{A}$ has products, let $N$ be a finitely cogenerated object of $\mathcal{A}$, and let
$(M_i)_{i\in I}$ be a family of objects of $\mathcal{A}$ such that $N$ is dual strongly $M_i$-Rickart for every $i\in I$. Then
$N$ is dual strongly $\prod_{i\in I} M_i$-Rickart if and only if $N$ is dual strongly $M_i$-Rickart for every $i\in I$.
\end{enumerate}
\end{coll}

\begin{proof} (1) The direct implication follows by Corollary~\ref{c1:summand}. For the converse, let $f:M\to
\bigoplus_{i\in I} N_i$ be a morphism in $\mathcal{A}$. Since $M$ is finitely generated, we may write
$f=jf'$ for some morphism $f':M\to \bigoplus_{i\in F} N_i$ and inclusion morphism $j:\bigoplus_{i\in F} N_i\to
\bigoplus_{i\in I} N_i$, where $F$ is a finite subset of $I$. By Theorem~\ref{t1:ds}, $\bigoplus_{i\in F} N_i$ is
strongly $M$-Rickart. Then ${\rm ker}(f)={\rm ker}(f')$ is a fully invariant section. 
Hence $\bigoplus_{i\in I} N_i$ is strongly $M$-Rickart.
\end{proof}

\begin{ex} \label{e1:abc} \rm (a) Theorem~\ref{t1:ds} does not hold in general for arbitrary coproducts. 
Indeed, let $R$ be a strongly regular ring such that there exists a left ideal $I$ which is not finitely generated
(i.e., $R$ is a strongly regular ring which is not semisimple). Then $R$ is strongly $R$-regular, 
hence $R$ is dual strongly $R$-Rickart. Let $f:R^{(I)}\to I$ be the canonical epimorphism, 
and let $i:I\to R$ be the inclusion homomorphism. If $R$ is dual strongly $R^{(I)}$-Rickart, then ${\rm
coker}(if):R\to {\rm Coker}(if)$ is a fully coinvariant retraction, 
hence $I={\rm Im}(f)={\rm Im}(if)$ is a fully invariant direct summand of $R$, and so it
is finitely generated, a contradiction. Hence $R$ is not dual strongly $R^{(I)}$-Rickart.

(b) Theorem~\ref{t1:ds} does not hold in general for (dual) strongly self-Rickart objects.
Indeed, consider the ring $R=\begin{pmatrix}\mathbb{Z}&\mathbb{Z}\\0&\mathbb{Z}\end{pmatrix}$ and the right $R$-modules
$M_1=\begin{pmatrix}\mathbb{Z}&\mathbb{Z}\\0&0\end{pmatrix}$ and $M_2=\begin{pmatrix}0&0\\0&\mathbb{Z}\end{pmatrix}$.
Then $M_1$ and $M_2$ are self-Rickart right $R$-modules \cite[Example~1.2]{LRR12}. Since $M_1$ and $M_2$ are indecomposable, 
they are strongly self-Rickart by Corollary \ref{c1:wduo}. But $R=M_1\oplus M_2$ is not a self-Rickart right $R$-module 
\cite[Example~1.2]{LRR12}, and so it is not strongly self-Rickart.

(c) One cannot change the role of (dual) relative Rickart objects in Theorem~\ref{t1:ds}.
Indeed, $R=\mathbb{Z}[X]$ is self-Rickart \cite[Example~2.11]{LRR12}. Since $R$ is indecomposable, 
it is strongly self-Rickart by Corollary \ref{c1:indec}. But $N$ is not an $M_1\oplus M_2$-Rickart $R$-module 
\cite[Example~2.11]{LRR12}, and so it is not strongly $M_1\oplus M_2$-Rickart.
\end{ex}

The next result gives a necessary condition for an infinite (co)product of objects to be a (dual) strongly self-Rickart object.

\begin{prop} \label{p1:relrickart} Let $(M_i)_{i\in I}$ be a family of objects of an abelian category $\mathcal{A}$.
\begin{enumerate} 
\item Assume that $\prod_{i\in I} M_i$ is a strongly self-Rickart object. Then $M_i$ is strongly $M_j$-Rickart for every $i,j\in I$.
\item Assume that $\bigoplus_{i\in I} M_i$ is a dual strongly self-Rickart object. Then $M_i$ is dual strongly $M_j$-Rickart for every
$i,j\in I$.
\end{enumerate}
\end{prop}

\begin{proof} (1) This follows by Theorem~\ref{t1:epimono}.
\end{proof}

Let $\mathcal{A}$ be an abelian category. Recall that an object $M$ of $\mathcal{A}$ has the \emph{strong summand
intersection property}, briefly \emph{SSIP}, if the intersection of any family of direct summands of $M$ 
is a direct summand of $M$ \cite{Wilson}. Dually, an object $M$ of $\mathcal{A}$ has the \emph{strong
summand sum property}, briefly \emph{SSSP}, if the sum of any family of direct summands of $M$ 
is a direct summand of $M$ \cite{Garcia}.

\begin{theo} \label{t1:SP} Let $\mathcal{A}$ be an abelian category.
\begin{enumerate}
\item Let $M$ be a weak duo object of $\mathcal{A}$ having SSIP, and let $(N_i)_{i\in I}$ be a family of objects of
$\mathcal{A}$ having a product. Then $\prod_{i\in I} N_i$ is
strongly $M$-Rickart if and only if $N_i$ is strongly $M$-Rickart for every $i\in I$.
\item Let $(M_i)_{i\in I}$ be a family of objects of $\mathcal{A}$ having a coproduct, and let $N$ be a weak duo object of
$\mathcal{A}$ having SSSP. Then $N$ is dual strongly $\bigoplus_{i\in I} M_i$-Rickart if and only if 
$N$ is dual strongly $M_i$-Rickart for every $i\in I$.
\end{enumerate}
\end{theo}

\begin{proof} (1) If $\prod_{i\in I} N_i$ is strongly $M$-Rickart, then $N_i$ is strongly $M$-Rickart for every $i\in I$ by Theorem
\ref{t1:epimono}. Conversely, assume that $N_i$ is strongly $M$-Rickart for every $i\in I$. Let $f:M\to \prod_{i\in I} N_i$
be a morphism in $\mathcal{A}$. For every $i\in I$, denote by $p_i:\prod_{i\in I} N_i\to N_i$ the canonical projection
and $f_i=p_if:M\to N_i$. Since $N_i$ is strongly $M$-Rickart, ${\rm Ker}(f_i)$ is a 
fully invariant direct summand of $M$ for every $i\in I$. Since $M$ is weak duo and has SSIP, 
it follows that ${\rm Ker}(f)=\bigcap_{i\in I} {\rm Ker}(f_i)$ is a fully invariant direct summand of $M$. 
Hence $\prod_{i\in I} N_i$ is strongly $M$-Rickart.
\end{proof}

\begin{theo} \label{t1:homzero} 
Let $\mathcal{A}$ be an abelian category. Let $M=\bigoplus_{i\in I}M_i$ be a direct sum of objects of $\mathcal{A}$. 
\begin{enumerate}
\item Then $M$ is strongly self-Rickart if and only if $M_i$ is strongly self-Rickart for every $i\in I$, 
and ${\rm Hom}_{\mathcal{A}}(M_i,M_j)=0$ for every $i,j\in I$ with $i\neq j$.
\item Then $M$ is dual strongly self-Rickart if and only if $M_i$ is dual strongly self-Rickart for every $i\in I$, 
and ${\rm Hom}_{\mathcal{A}}(M_i,M_j)=0$ for every $i,j\in I$ with $i\neq j$.
\end{enumerate}
\end{theo}

\begin{proof} (1) Assume that $M$ is strongly self-Rickart. Then $M_i$ is a fully invariant direct summand of $M$ 
for every $i\in I$. It follows that ${\rm Hom}_{\mathcal{A}}(M_i,M_j)=0$ for every $i,j\in I$ with $i\neq j$. 
Also, $M_i$ is strongly self-Rickart for every $i\in I$ by Corollary \ref{c1:summand}.

Conversely, assume that $M_i$ is strongly self-Rickart for every $i\in I$, 
and ${\rm Hom}_{\mathcal{A}}(M_i,M_j)=0$ for every $i,j\in I$ 
with $i\neq j$. Let $f:M\to M$ be a morphism in $\mathcal{A}$. Its associated matrix has zero entries except for 
the entries $(i,i)$ with $i\in I$, which are some morphisms $f_i:M_i\to M_i$. It follows that $f=\bigoplus_{i\in I}f_i$. 
Then $K={\rm Ker}(f)=\bigoplus_{i\in I}{\rm Ker}(f_i)$. Since $M_i$ is strongly self-Rickart, 
$k_i:K_i={\rm Ker}(f_i)\to M_i$ is a fully invariant section for every $i\in I$. For every $i\in I$,
let $p_i:M_i\to K_i$ be the canonical projection. Then for every $i\in I$, we have $p_ik_i=1_{K_i}$ and 
$e_i=k_ip_i\in {\rm End}_{\mathcal{A}}(M_i)$ is an idempotent. Then $e_i$ is a central idempotent 
for every $i\in I$ by Proposition \ref{p1:strring}. We have $k=\bigoplus_{i\in I}k_i:K\to M$ and 
let $p=\bigoplus_{i\in I}p_i:M\to K$. Then $pk=1_K$ and 
$e=kp=\bigoplus_{i\in I}k_ip_i=\bigoplus_{i\in I}e_i\in {\rm End}_{\mathcal{A}}(M)$ is a central idempotent. 
It follows that $k$ is a fully invariant section by Lemma \ref{l1:semic}. Hence $M$ is strongly self-Rickart. 

(2) This is not completely dual to (1), but it follows in a similar way as (1) by using images instead of kernels.
\end{proof}

\begin{ex} \rm The indecomposable abelian groups $\mathbb{Z}$ and $\mathbb{Z}_p$ for some prime $p$ are strongly self-Rickart. 
By Theorem \ref{t1:homzero}, $\mathbb{Z}\oplus \mathbb{Z}_p$ is not strongly self-Rickart, 
because ${\rm Hom}_{\mathbb{Z}}(\mathbb{Z},\mathbb{Z}_p)\neq 0$.  
\end{ex}

The above properties allow us to give the following results on the structure of (dual) strongly self-Rickart modules 
over a Dedekind domain, and in particular, (dual) strongly self-Rickart abelian groups.

\begin{coll} \label{c1:dede}
Let $R$ be a Dedekind domain with quotient field $K$. 
\begin{enumerate}
\item 
\begin{enumerate}[(i)] 
\item A non-zero torsion $R$-module $M$ is strongly self-Rickart if and only if 
$M\cong \bigoplus_{i\in I} R/P_i$ for some distinct maximal ideals $P_i$ of $R$.
\item A finitely generated $R$-module $M$ is strongly self-Rickart if and only if 
$M\cong J$ for some ideal $J$ of $R$ or $M\cong \bigoplus_{i=1}^kR/P_i$ 
for some distinct maximal ideals $P_i$ of $R$.
\item A non-zero injective $R$-module $M$ is strongly self-Rickart if and only if $M\cong K$.
\end{enumerate}
\item
\begin{enumerate}[(i)] 
\item A non-zero torsion $R$-module $M$ is dual strongly self-Rickart if and only if 
$M\cong \bigoplus_{i\in I} E(R/P_i)$ or $M\cong \bigoplus_{i\in I} R/P_i$ for some distinct maximal ideals $P_i$ of $R$.
\item A non-zero finitely generated $R$-module $M$ is dual strongly self-Rickart if and only if 
$M\cong \bigoplus_{i=1}^kR/P_i$ for some distinct maximal ideals $P_i$ of $R$.
\item A non-zero injective $R$-module $M$ is dual strongly self-Rickart if and only if 
$M\cong \bigoplus_{i\in I} M_i$ for some distinct $R$-modules $M_i$ which are either $K$ or 
$E(R/P_i)$ for some maximal ideal $P_i$ of $R$.
\end{enumerate}
\end{enumerate}
\end{coll}

\begin{proof} By Corollary \ref{c1:wduo}, $M$ is strongly self-Rickart if and only if $M$ is self-Rickart and weak duo, 
while $M$ is dual strongly self-Rickart if and only if $M$ is dual self-Rickart and weak duo. By Corollary \ref{c1:summand},
every direct summand of a (dual) strongly self-Rickart module is (dual) strongly self-Rickart.

We use the following structure theorems on weak duo modules over a Dedekind domain $R$: 
(i) a non-zero torsion $R$-module $M$ is (weak) duo if and only if $M\cong \bigoplus_{i\in I} E(R/P_i)$ or 
$M\cong \bigoplus_{i\in I} R/P_i^{n_i}$ for some distinct maximal ideals $P_i$ of $R$ and positive integers $n_i$ \cite[Theorem~3.10]{OHS}; 
(ii) a finitely generated $R$-module $M$ is (weak) duo if and only if $M\cong J$ for some ideal $J$ of $R$ or 
$M\cong \bigoplus_{i=1}^kR/P_i^{n_i}$ for some distinct maximal ideals $P_1,\dots,P_k$ of $R$ 
and positive integers $n_1,\dots,n_k$ \cite[Corollary~3.11]{OHS}.

Let $P$ be a maximal ideal of $R$ and $n$ a positive integer. 

(1) (i) By \cite[Theorem~5.6]{LRR10}, the torsion $R$-module $R/P^n$ is self-Rickart if and only if $n=1$ if and only if 
it is simple. Now for any distinct maximal ideals $P_i$ of $R$, $\bigoplus_{i\in I} R/P_i$ is self-Rickart, 
as a semisimple $R$-module.

If $E(R/P)=R/P$, then it is self-Rickart by the first part. Now assume that $E(R/P)\neq R/P$. 
Consider $f\in {\rm End}_R(E(R/P))$ defined by $f(a)=pa$ for some $0\neq p\in P$. 
If $f=0$, then $p\in {\rm Ann}_R(E(R/P))=0$ \cite[Proposition~2.26, Corollary~1]{SV}, which is a contradiction. 
Hence $f\neq 0$, and so ${\rm Ker}(f)\neq E(R/P)$. Let $0\neq a\in E(R/P)$. 
Since $R/P$ is an essential submodule of $E(R/P)$, there exists $r\in R$ such that $0\neq ra\in R/P$. Then $pra=0$, and so 
$0\neq ra\in {\rm Ker}(f)$. Since $E(R/P)$ is uniform, ${\rm Ker}(f)$ is a non-zero proper essential submodule of $E(R/P)$. 
This shows that $E(R/P)$ is not self-Rickart. Now the conclusion follows.

(ii) Since $R$ is hereditary, every submodule of a projective $R$-module is projective, and every projective $R$-module is 
self-Rickart \cite[Theorem~2.26]{LRR10}. Hence every ideal $J$ of $R$ is self-Rickart. 

For any distinct maximal ideals $P_1,\dots,P_k$ of $R$ and positive integers $n_1,\dots,n_k$, 
the direct sum $\bigoplus_{i=1}^kR/P_i^{n_i}$ is self-Rickart if and only if $n_i=1$ for every $i\in \{1,\dots,k\}$. 
Now the conclusion follows.

(iii) By \cite[Theorem~7]{K52} every injective module $M$ over a Dedekind domain $R$ is of the form 
$M=\bigoplus_{i\in I}M_i$, where $M_i$ is either a vector space over $K$ or $E(R/P_i)$ for some prime ideal $P_i$ of $R$. 
Combining it with the structure theorem of non-zero torsion weak duo $R$-modules, 
the fact proved in (i) that $E(R/P)$ is not self-Rickart, and Theorem \ref{t1:homzero}, one deduces that 
the only left possibility is $M\cong K$. But $K$ is clearly strongly self-Rickart.

(2) (i) Let $(P_i)_{i\in I}$ be a family of distinct maximal ideals of $R$. Since $R$ is noetherian, 
$\bigoplus_{i\in I} E(R/P_i)$ is injective. Since $R$ is hereditary, every injective $R$-module is 
dual self-Rickart by \cite[Theorem~2.29]{LRR11}. Hence so is $\bigoplus_{i\in I} E(R/P_i)$.

We claim that the torsion $R$-module $R/P^n$ is dual self-Rickart if and only if $n=1$ if and only if it is simple.
For the non-trivial implication, similarly to the proof of \cite[Theorem~5.6]{LRR10}, 
let $n>1$ and consider $f\in {\rm End}_R(R/P^n)$ defined by $f(\hat{x})=p\hat{x}$ for some $p\in P\setminus P^2$, 
where $\hat{x}=x+P^n$. Then $\hat{p}^{n-1}\in {\rm Ker}(f)\neq \hat{0}$, and so 
${\rm Im}(f)\neq P^n$. Also, ${\rm Im}(f)\neq \hat{0}$, because $f(\hat{1})=\hat{p}\neq \hat{0}$. 
Since $R/P^n$ is uniform, ${\rm Im}(f)$ is a non-zero proper essential submodule of $R/P^n$. 
This shows that $R/P^n$ is not dual self-Rickart.
Consequently, one must have $n=1$. For any distinct maximal ideals $P_i$ of $R$, $\bigoplus_{i\in I} R/P_i$ is dual self-Rickart, 
as a semisimple $R$-module. Now the conclusion follows.

(ii) Since $R$ is noetherian, the conclusion follows by \cite[Theorem~4.14]{LRR11}.

(iii) Using the structure theorems of injective modules and torsion weak duo modules over a Dedekind domain,
we deduce that $M$ must have the form $M\cong \bigoplus_{i\in I} M_i$ for some distinct $R$-modules $M_i$ which are either $K$ or 
$E(R/P_i)$ for some maximal ideal $P_i$ of $R$. But $K$ is clearly dual self-Rickart, which together with (i) imply 
the conclusion.
\end{proof}

\begin{coll} \label{c1:abgr}
\begin{enumerate}
\item 
\begin{enumerate}[(i)] 
\item A non-zero torsion abelian group $G$ is strongly self-Rickart if and only if 
$G\cong \bigoplus_{i\in I} \mathbb{Z}_{p_i}$ for some distinct primes $p_i$.
\item A finitely generated abelian group $G$ is strongly self-Rickart if and only if 
$G\cong \mathbb{Z}$ or $G\cong \bigoplus_{i=1}^k \mathbb{Z}_{p_i}$ for some distinct primes $p_i$.
\item A non-zero injective abelian group $G$ is strongly self-Rickart if and only if $G\cong \mathbb{Q}$.
\end{enumerate}
\item
\begin{enumerate}[(i)] 
\item A non-zero torsion abelian group $G$ is dual strongly self-Rickart if and only if 
$G\cong \bigoplus_{i\in I} \mathbb{Z}_{p_i^{\infty}}$ or $G\cong \bigoplus_{i\in I} \mathbb{Z}_{p_i}$ 
for some distinct primes $p_i$.
\item A non-zero finitely generated abelian group $G$ is dual strongly self-Rickart if and only if 
$G\cong \bigoplus_{i=1}^k \mathbb{Z}_{p_i}$ for some distinct primes $p_i$.
\item A non-zero injective abelian group $G$ is dual strongly self-Rickart if and only if 
$G\cong \bigoplus_{i\in I} G_i$ for some distinct groups $G_i$  
which are either $\mathbb{Q}$ or $\mathbb{Z}_{p_i^{\infty}}$ for some primes $p_i$.
\end{enumerate}
\end{enumerate}
\end{coll}

Corollary \ref{c1:abgr} allows one to easily construct examples of (dual) self-Rickart abelian groups, 
which are not (dual) strongly self-Rickart.

\begin{ex} \label{e1:abgr} \rm The abelian group $\mathbb{Z}_p\oplus \mathbb{Z}_p$ (for some prime $p$) is 
both self-Rickart and dual self-Rickart, being semisimple. But it is neither strongly self-Rickart, 
nor dual strongly self-Rickart by Corollary~\ref{c1:abgr}.
\end{ex}

\section{(Dual) strongly relative Rickart objects: transfer via functors}

Our first result on the transfer of the (dual) strongly relative Rickart properties via (additive) functors 
involves a fully faithful covariant functor. 

\begin{theo} \label{t1:ff} Let $F:\mathcal{A}\to \mathcal{B}$ be a fully faithful covariant functor between abelian
categories, and let $M$ and $N$ be objects of $\mathcal{A}$. 
\begin{enumerate}
\item Assume that $F$ is left exact. Then $N$ is strongly $M$-Rickart in $\mathcal{A}$ if and only if $F(N)$ is strongly $F(M)$-Rickart
in $\mathcal{B}$.
\item Assume that $F$ is right exact. Then $N$ is dual strongly $M$-Rickart in $\mathcal{A}$ if and only if $F(N)$ is dual strongly 
$F(M)$-Rickart in $\mathcal{B}$.
\end{enumerate}
\end{theo}

\begin{proof} (1) Assume that $N$ is strongly $M$-Rickart. Let $g:F(M)\to F(N)$ be a morphism in $\mathcal{B}$ with kernel 
$l={\rm ker}(g):L\to F(M)$. Since $F$ is full, we have $g=F(f)$ for some morphism 
$f:M\to N$ in $\mathcal{A}$ with kernel $k={\rm ker}(f):K\to M$. 
Since $N$ is strongly $M$-Rickart, $k$ is a fully invariant section. Clearly, $F(k)$ is a section. 
Since $F$ is left exact, we have $l={\rm ker}(F(f))=F(k)$. Hence $l=F(k)$ is a section. 

Now let $\varphi:F(M)\to F(M)$ be a morphism in $\mathcal{B}$. Since $F$ is full, we have $\varphi=F(h)$ 
for some morphism $h:M\to M$ in $\mathcal{A}$. Since $k$ is a fully invariant kernel, we have 
$hk=k\alpha$ for some morphism $\alpha:K\to K$. It follows that 
\[\varphi l=F(h)F(k)=F(hk)=F(k\alpha)=F(k)F(\alpha)=lF(\alpha).\] 
Hence $l$ is a fully invariant kernel. This shows that $F(N)$ is strongly $F(M)$-Rickart.

Conversely, assume that $F(N)$ is strongly $F(M)$-Rickart. Let $f:M\to N$ be a morphism in $\mathcal{A}$ with kernel $k:K\to M$.
Since $F(N)$ is strongly $F(M)$-Rickart, $l={\rm ker}(F(f))$ is a fully invariant section. 
Since $F$ is left exact, we have $l={\rm ker}(F(f))=F(k)$. Since $F$ is fully faithful, it reflects sections. 
Hence ${\rm ker}(f)=k$ is a section. 

Now let $h:M\to M$ be a morphism in $\mathcal{A}$. Then $F(h):F(M)\to F(M)$ in $\mathcal{B}$. Since 
$l={\rm ker}(F(f))$ is a fully invariant section, we have 
$F(h)l=l\beta$ for some morphism $\beta:L\to L$. Since $F$ is full, we have $\beta=F(\alpha)$ for some morphism $\alpha:K\to K$. 
It folllows that \[F(hk)=F(h)F(k)=F(k)F(\alpha)=F(k\alpha),\] hence $hk=k\alpha$, because $F$ is faithful.
Thus, $k$ is a fully invariant section. This shows that $N$ is strongly $M$-Rickart. 
\end{proof}

For Grothendieck categories we have the following corollary.

\begin{coll} \label{c1:gp} Let $\mathcal{A}$ be a Grothendieck category with a generator $U$, $R={\rm
End}_{\mathcal{A}}(U)$, $S={\rm Hom}_{\mathcal{A}}(U,-):\mathcal{A}\to {\rm Mod}(R)$, and let $M$ and $N$ be objects
of $\mathcal{A}$. Then $N$ is a strongly $M$-Rickart object of $\mathcal{A}$ if and only if $S(N)$ is a strongly $S(M)$-Rickart right
$R$-module.
\end{coll}

\begin{proof} By the Gabriel-Popescu Theorem \cite[Chapter~X, Theorem~4.1]{St}, $S$ is a fully faithful functor which
has a left adjoint $T:{\rm Mod}(R)\to \mathcal{A}$. Since $S$ is a right adjoint, it is left exact. Then the
conclusion follows by Theorem \ref{t1:ff}.  
\end{proof}

By an \emph{adjoint triple} of functors we mean a triple $(L,F,R)$ of covariant functors $F:\mathcal{A}\to \mathcal{B}$
and $L,R:\mathcal{B}\to \mathcal{A}$ such that $(L,F)$ and $(F,R)$ are adjoint pairs of functors. Then $F$ is an exact
functor as a left and right adjoint. Note that $L$ is fully faithful if and only if $R$ is fully faithful
\cite[Lemma~1.3]{DT}. A particular instance of adjoint triple $(L,F,R)$ of functors is when $L=R$, 
in which case $(F,R)$ is called a \emph{Frobenius pair} \cite{CGN}.

Now Theorem \ref{t1:ff} yields the following consequence.  

\begin{coll} \label{c1:tripleff} Let $(L,F,R)$ be an adjoint triple of covariant functors $F:\mathcal{A}\to \mathcal{B}$
and $L,R:\mathcal{B}\to \mathcal{A}$ between abelian categories. 
\begin{enumerate}
\item Let $M$ and $N$ be objects of $\mathcal{A}$, and assume that $F$ is fully faithful. Then $N$ is (dual) strongly 
$M$-Rickart in $\mathcal{A}$ if and only if $F(N)$ is (dual) strongly $F(M)$-Rickart in $\mathcal{B}$.
\item Let $M$ and $N$ be objects of $\mathcal{B}$, and assume that $L$ (or $R$) is fully faithful. Then: 
\begin{enumerate}[(i)] 
\item $N$ is strongly $M$-Rickart in $\mathcal{B}$ if and only if $R(N)$ is strongly $R(M)$-Rickart in $\mathcal{A}$.
\item $N$ is dual strongly $M$-Rickart in $\mathcal{B}$ if and only if $L(N)$ is dual strongly $L(M)$-Rickart in $\mathcal{A}$.
\end{enumerate}
\end{enumerate}
\end{coll}

Let $\varphi:R\to S$ be a ring homomorphism. Following \cite[Chapter~IX, p.105]{St}, consider the following covariant
functors: extension of scalars $\varphi^*:{\rm Mod}(R)\to {\rm Mod}(S)$ given on objects by $\varphi^*(M)=M\otimes_RS$,
restriction of scalars $\varphi_*:{\rm Mod}(S)\to {\rm Mod}(R)$ given on objects by $\varphi_*(N)=N$, and
$\varphi^!:{\rm Mod}(R)\to {\rm Mod}(S)$ given on objects by $\varphi^!(M)={\rm Hom}_R(S,M)$. Then
$(\varphi^*,\varphi_*,\varphi^!)$ is an adjoint triple of functors. 

\begin{coll} Let $\varphi:R\to S$ be a ring epimorphism, and let $M$ and $N$ be right $S$-modules. Then $N$ is a (dual)
strongly $M$-Rickart right $S$-module if and only if $N$ is a (dual) strongly $M$-Rickart right $R$-module.
\end{coll}

\begin{proof} Since $\varphi:R\to S$ is a ring epimorphism, the restriction of scalars functor $\varphi_*:{\rm
Mod}(S)\to {\rm Mod}(R)$ is fully faithful \cite[Chapter~XI, Proposition~1.2]{St}. Then use Corollary \ref{c1:tripleff}
for the adjoint triple of functors $(\varphi^*,\varphi_*,\varphi^!)$.
\end{proof}

Now let us recall some terminology and properties of graded rings and modules, following \cite{Nasta-04}. 
In what follows $G$ will denote a group with identity element $e$. A ring $R$ is called \emph{(strongly) $G$-graded} 
if there is a family $(R_{\sigma})_{\sigma\in G}$ of additive subgroups of $R$ such that 
$R=\bigoplus_{\sigma\in G}R_{\sigma}$ and $R_{\sigma}R_{\tau}\subseteq R_{\sigma\tau}$ 
($R_{\sigma}R_{\tau}=R_{\sigma\tau}$) for every $\sigma,\tau\in G$. 

For a $G$-graded ring $R=\bigoplus_{\sigma\in G}R_{\sigma}$, denote by ${\rm gr}(R)$ the category which
has as objects the $G$-graded unital right $R$-modules and as morphisms the morphisms of $G$-graded unital right
$R$-modules, defined as follows. For a $G$-graded ring $R=\bigoplus_{\sigma\in G}R_{\sigma}$, a \emph{$G$-graded} (for
short, \emph{graded}) right $R$-module is a right $R$-module $M$ such that $M=\bigoplus_{\sigma\in G}M_{\sigma}$, where
every $M_{\sigma}$ is an additive subgroup of $M$ and for every $\lambda,\sigma\in G$, we have $M_{\sigma}\cdot
R_{\lambda}\subseteq M_{\sigma\lambda}$. For two graded right $R$-modules $M$ and $N$, the morphisms between them
are defined as follows: $$\hom{\mathrm{gr}(R)}{M}{N}=\{f\in \hom{R}{M}{N}\mid f(M_{\sigma})\subseteq N_{\sigma} \textrm{
for every } \sigma\in G\}.$$ Note that $\mathrm{gr}(R)$ is a Grothendieck category (e.g., see \cite[p.~21]{Nasta-04}).

Recall that for a graded right $R$-module $M$ and $\sigma\in G$, the \emph{$\sigma$-suspension} $M(\sigma)$ of $M$ 
is the graded right $R$-module which is equal to $M$ as a right $R$-module and the gradation is given by 
$M({\sigma})_{\tau}=M_{\tau\sigma}$ for every $\tau\in G$. 
Then the \emph{$\sigma$-suspension functor} $T_{\sigma}:\mathrm{gr}(R)\to \mathrm{gr}(R)$ is defined by 
$T_{\sigma}(M)=M({\sigma})$. 

For $\sigma\in G$, define the functor $(-)_{\sigma}:\mathrm{gr}(R)\to {\rm Mod}(R_e)$, 
which associates to every graded right $R$-module $M=\bigoplus_{\tau\in G}M_{\tau}$ the right $R_e$-module $M_{\sigma}$.  
The \emph{induced functor} ${\rm Ind}:{\rm Mod}(R_e)\to {\rm gr}(R)$ is defined as follows. 
For a right $R_e$-module $N$, ${\rm Ind}(N)$ is the graded right $R$-module $M=N\otimes_{R_e}R$, 
where the gradation of $M=\bigoplus_{\sigma\in G}M_{\sigma}$ is given by 
$M_{\sigma}=N_{\sigma}\otimes_{R_e}R$ for every $\sigma\in G$. 
The \emph{coinduced functor} ${\rm Coind}:{\rm Mod}(R_e)\to {\rm gr}(R)$ is defined as follows.
For a right $R_e$-module $N$, ${\rm Coind}(N)$ is the graded right $R$-module 
$M^*=\bigoplus_{\sigma\in G}M'_{\sigma}$, where 
\[M'_{\sigma}=\{f\in {\rm Hom}_{R_e}(R,N)\mid f(R_{\sigma'})=0 \textrm{ for every } \sigma'\neq \sigma^{-1}\}.\]
Then $(T_{\sigma^{-1}}\circ {\rm Ind},(-)_{\sigma},T_{\sigma^{-1}}\circ {\rm Coind})$ is an adjoint triple of functors 
\cite[Theorem~2.5.5]{Nasta-04}. In particular, $({\rm Ind},(-)_{e},{\rm Coind})$ is an adjoint triple of functors.

\begin{coll} \label{c1:triplegr} Let $R$ be a $G$-graded ring, and let $M$ and $N$ be right $R_e$-modules. Then: 
\begin{enumerate} 
\item $N$ is a strongly $M$-Rickart right $R_e$-module if and only if 
${\rm Coind}(N)$ is a strongly ${\rm Coind}(M)$-Rickart graded right $R$-module.
\item $N$ is a dual strongly $M$-Rickart right $R_e$-module if and only if 
${\rm Ind}(N)$ is a dual strongly ${\rm Ind}(M)$-Rickart graded right $R$-module.
\end{enumerate}
\end{coll}

\begin{proof} By \cite[Theorem~2.5.5]{Nasta-04}, for $\sigma\in G$, 
we have $(-)_{\sigma}\circ T_{\sigma^{-1}}\circ {\rm Ind}\cong 1_{{\rm Mod}(R_e)}$ 
and $(-)_{\sigma}\circ T_{\sigma^{-1}}\circ {\rm Coind}\cong 1_{{\rm Mod}(R_e)}$, that is,  
the functors $T_{\sigma^{-1}}\circ {\rm Ind}$ and $T_{\sigma^{-1}}\circ {\rm Coind}$ are fully faithful. 
In particular, the functors ${\rm Ind}$ and ${\rm Coind}$ are fully faithful.
Now use Corollary \ref{c1:tripleff}. 
\end{proof}

Let $\mathcal{A}$ be an abelian category and let $\mathcal{C}$ be a full subcategory of $\mathcal{A}$. Then
$\mathcal{C}$ is called a \emph{reflective} (\emph{coreflective}) subcategory of $\mathcal{A}$ if the inclusion functor
$i:\mathcal{C}\to \mathcal{A}$ has a left (right) adjoint. In this case $i$ is fully faithful. 

\begin{coll} \label{c1:rc} Let $\mathcal{A}$ be an abelian category, $\mathcal{C}$ an abelian full subcategory of
$\mathcal{A}$ and $i:\mathcal{C}\to \mathcal{A}$ the inclusion functor. Let $M$ and $N$ be objects of $\mathcal{C}$. 
\begin{enumerate}
\item Assume that $\mathcal{C}$ is a reflective subcategory of $\mathcal{A}$. 
Then $N$ is strongly $M$-Rickart in $\mathcal{C}$ if and only if $i(N)$ is strongly $i(M)$-Rickart in $\mathcal{A}$.
\item Assume that $\mathcal{C}$ is a coreflective subcategory of $\mathcal{A}$. 
Then $N$ is (dual) strongly $M$-Rickart in $\mathcal{C}$ if and only if $i(N)$ is (dual) strongly $i(M)$-Rickart in $\mathcal{A}$.
\end{enumerate}
\end{coll}

\begin{proof} (1) Note that $i$ is left exact fully faithful and use Theorem \ref{t1:ff}.

(2) Note that $i$ is exact fully faithful and use Theorem \ref{t1:ff}.
\end{proof}

Following \cite[Section~2.2]{DNR}, let $C$ be a coalgebra over a field $k$, and let ${}^C\mathcal{M}$ be the
(Grothendieck) category of left $C$-comodules. Left $C$-comodules may be and will be identified with rational right
$C^*$-modules, where $C^*={\rm Hom}_{k}(C,k)$. Consider the inclusion functor $i:{}^C\mathcal{M}\to {\rm Mod}(C^*)$,
and the functor ${\rm Rat}:{\rm Mod}(C^*)\to {}^C\mathcal{M}$ which associates to every right $C^*$-module its rational
$C^*$-submodule. Then $i$ is a fully faithful exact functor and ${\rm Rat}$ is a right adjoint to $i$. Hence
${}^C\mathcal{M}$ is a coreflective subcategory of ${\rm Mod}(C^*)$. 

\begin{coll} \label{c1:com1} Let $C$ be a coalgebra over a field, and let $M$ and $N$ be left $C$-comodules. 
Then $N$ is (dual) strongly $M$-Rickart if and only if $N$ is (dual) strongly $M$-Rickart as a right $C^*$-module. 
\end{coll}

In order to discuss the transfer of the (dual) strong relative Rickart property to endomorphism rings, 
we establish first some general results involving adjoint functors. We need the following notions.

\begin{defn}[{\cite[Definition~2.10]{CO}}] \rm Let $M$ and $N$ be objects of an abelian category $\mathcal{A}$. Then: 
\begin{enumerate}
\item $N$ is called \emph{$M$-cyclic} if there exists an epimorphism $M\to N$.
\item $M$ is called \emph{$N$-cocyclic} if there exists a monomorphism $M\to N$.
\end{enumerate}
\end{defn}

Let $(L,R)$ be an adjoint pair of covariant functors $L:\mathcal{A}\to \mathcal{B}$ and $R:\mathcal{B}\to \mathcal{A}$
between abelian categories. Let $\varepsilon:LR\to 1_{\mathcal{B}}$ and $\eta:1_{\mathcal{A}}\to RL$ be the counit and
the unit of adjunction respectively. Recall that an object $B\in \mathcal{B}$ is called \emph{$R$-static} if
$\varepsilon_B$ is an isomorphism, while an object $A\in \mathcal{A}$ is called \emph{$R$-adstatic} if $\eta_A$ is an
isomorphism \cite{CGW}. Denote by ${\rm Stat}(R)$ the full subcategory of $\mathcal{B}$ consisting of $R$-static
objects, and by ${\rm Adst}(R)$ the full subcategory of $\mathcal{A}$ consisting of $R$-adstatic objects. 

\begin{theo} \label{t1:equiv} Let $(L,R)$ be an adjoint pair of covariant functors $L:\mathcal{A}\to \mathcal{B}$ and
$R:\mathcal{B}\to \mathcal{A}$ between abelian categories. 
\begin{enumerate}
\item Let $M$ and $N$ be objects of $\mathcal{B}$ such that $M,N\in {\rm Stat}(R)$. Then the following are equivalent:
\begin{enumerate}[(i)] 
\item $N$ is strongly $M$-Rickart in $\mathcal{B}$.
\item $R(N)$ is strongly $R(M)$-Rickart in $\mathcal{A}$ and for every morphism $f:M\to N$, ${\rm Ker}(f)$ is $M$-cyclic.
\item $R(N)$ is strongly $R(M)$-Rickart in $\mathcal{A}$ and for every morphism $f:M\to N$, ${\rm Ker}(f)\in {\rm Stat}(R)$.
\end{enumerate}
\item Let $M$ and $N$ be objects of $\mathcal{A}$ such that $M,N\in {\rm Adst}(R)$. Then the following are equivalent:
\begin{enumerate}[(i)] 
\item $N$ is dual strongly $M$-Rickart in $\mathcal{A}$.
\item $L(N)$ is dual strongly $L(M)$-Rickart in $\mathcal{B}$ and for every morphism $f:M\to N$, ${\rm Coker}(f)$ is
$N$-cocyclic.
\item $L(N)$ is dual strongly $L(M)$-Rickart in $\mathcal{B}$ and for every morphism $f:M\to N$, ${\rm Coker}(f)\in {\rm
Adst}(R)$.
\end{enumerate}
\end{enumerate}
\end{theo}

\begin{proof}  We use the idea of the proof of \cite[Theorem~2.11]{CO}.

Let $\varepsilon:LR\to 1_{\mathcal{B}}$ and $\eta:1_{\mathcal{A}}\to RL$ be the counit and the unit of
adjunction respectively.

(1) (i)$\Rightarrow$(ii) Assume that $N$ is strongly $M$-Rickart in $\mathcal{B}$. Let $g:R(M)\to R(N)$ be a morphism in
$\mathcal{A}$. By naturality we have the following commutative diagram in $\mathcal{A}$: 
\[\SelectTips{cm}{}
\xymatrix{
R(M) \ar[d]_{\eta_{R(M)}} \ar[r]^g & R(N) \ar[d]^{\eta_{R(N)}} \\ 
RLR(M) \ar[r]_{RL(g)} & RLR(N)  
} 
\]
Since $R(\varepsilon_M)\eta_{R(M)}=1_{R(M)}$ and $\varepsilon_M$ is an isomorphism, it follows that $\eta_{R(M)}$ is an
isomorphism. Similarly, $\eta_{R(N)}$ is an isomorphism. Then we have $RL(g)=g$. Viewing $L(g):M\to N$, 
$l={\rm ker}(L(g)):L\to M$ is a fully invariant section, because $N$ is strongly $M$-Rickart. But $R$ is left exact, hence 
$k={\rm ker}(g)={\rm ker}(RL(g))=R({\rm ker}(L(g)))=R(l)$ is a section. We claim that $k$ is a fully invariant section. 
To this end, let $\varphi:R(M)\to R(M)$ be a morphism. Viewing $L(\varphi):M\to M$, we have $L(\varphi)k=k\alpha$ 
for some morphism $\alpha:L\to L$, because $N$ is strongly $M$-Rickart. 
Then $RL(\varphi)R(l)=R(l)R(\alpha)$. But $\eta_{R(M)}$ is an isomorphism, 
hence we may identify $RL(\varphi)$ with $\varphi$. Then $\varphi k=\varphi R(l)=R(l)R(\alpha)=kR(\alpha)$.
Thus, $k$ is a fully invariant section. Then $R(N)$ is strongly $R(M)$-Rickart in $\mathcal{A}$. 
Finally, for every morphism $f:M\to N$ in $\mathcal{B}$, ${\rm Ker}(f)$ is a direct summand of $M$, hence ${\rm Ker}(f)$ is $M$-cyclic.

(ii)$\Rightarrow$(iii) Assume that $R(N)$ is strongly $R(M)$-Rickart in $\mathcal{A}$ and for every morphism $f:M\to
N$, ${\rm Ker}(f)$ is $M$-cyclic. Let $f:M\to N$ be a morphism in $\mathcal{B}$ with kernel $k={\rm
ker}(f):K\to M$. Since $R(N)$ is strongly $R(M)$-Rickart, ${\rm ker}(R(f))$ is a fully invariant section. Since $R$ is left
exact, it follows that $R(k)$ is a fully invariant section, and so $LR(k)$ is a section. 
We claim that $LR(k)$ is a fully invariant section. To this end, let $\varphi:LR(M)\to LR(M)$. 
Viewing $\varphi:M\to M$, we have $R(\varphi)R(k)=R(k)\alpha$ for some morphism $\alpha:R(K)\to R(K)$, 
because $R(N)$ is strongly $R(M)$-Rickart. Then $LR(\varphi)LR(k)=LR(k)L(\alpha)$. 
We may identify $LR(\varphi)$ with $\varphi$. Then $\varphi LR(k)=LR(k)L(\alpha)$. Thus $LR(k)$ is a 
fully invariant section. 

Since $K$ is $M$-cyclic, there is an epimorphism $p:M\to K$. 
By naturality we have the following two commutative diagrams in $\mathcal{B}$: 
\[\SelectTips{cm}{}
\xymatrix{
LR(K) \ar[d]_{\varepsilon_K} \ar[r]^{LR(k)} & LR(M) \ar[d]^{\varepsilon_M} \\ 
K \ar[r]_k & M  
} 
\hspace{2cm}
\xymatrix{
LR(M) \ar[d]_{\varepsilon_M} \ar[r]^{LR(p)} & LR(K) \ar[d]^{\varepsilon_K} \\ 
M \ar[r]_p & K  
} 
\]
Then $k\varepsilon_K=\varepsilon_M LR(k)$ is a section, hence $\varepsilon_K$ is a section. Also, $\varepsilon_K
LR(p)=p\varepsilon_M$ is an epimorphism, hence $\varepsilon_K$ is an epimorphism. Then $\varepsilon_K$ is
an isomorphism, which implies that ${\rm Ker}(f)\in {\rm Stat}(R)$. 

(iii)$\Rightarrow$(i) Assume that $R(N)$ is strongly $R(M)$-Rickart in $\mathcal{A}$ and for every morphism $f:M\to
N$, ${\rm Ker}(f)\in {\rm Stat}(R)$. Let $f:M\to N$ be a morphism in $\mathcal{B}$ with kernel $k={\rm
ker}(f):K\to M$. Then $\varepsilon_K$ and $\varepsilon_M$ are isomorphisms, and $LR(k)$ is a fully invariant section as above. Now
the above left hand side diagram implies that $k$ is a fully invariant section. Hence $N$ is strongly $M$-Rickart in $\mathcal{B}$.
\end{proof}

We need some further terminology in order to state the contravariant version of Theorem \ref{t1:equiv}, 
which will be useful when discussing the transfer of (dual) strongly self-Rickart property between a module 
and its endomorphism ring.

Let $(L,R)$ be a right adjoint pair of contravariant functors $L:\mathcal{A}\to
\mathcal{B}$ and $R:\mathcal{B}\to \mathcal{A}$ between abelian categories. Let $\varepsilon:1_{\mathcal{B}}\to LR$ and
$\eta:1_{\mathcal{A}}\to RL$ be the counit and the unit of adjunction respectively. Recall that an object $B\in
\mathcal{B}$ is called \emph{$R$-reflexive} if $\varepsilon_B$ is an isomorphism, while an object $A\in \mathcal{A}$
is called \emph{$L$-reflexive} if $\eta_A$ is an isomorphism \cite{Castano}. Denote by ${\rm Refl}(R)$ the full
subcategory of $\mathcal{B}$ consisting of $R$-reflexive objects, and by ${\rm Refl}(L)$ the full subcategory of
$\mathcal{A}$ consisting of $L$-reflexive objects. 

\begin{theo} \label{t1:dual} Let $(L,R)$ be a pair of contravariant functors $L:\mathcal{A}\to \mathcal{B}$ and
$R:\mathcal{B}\to \mathcal{A}$ between abelian categories. 
\begin{enumerate}
\item Assume that $(L,R)$ is left adjoint. Let $M$ and $N$ be objects of $\mathcal{B}$ such that $M,N\in {\rm Refl}(R)$.
Then the following are equivalent:
\begin{enumerate}[(i)] 
\item $N$ is strongly $M$-Rickart in $\mathcal{B}$.
\item $R(M)$ is dual strongly $R(N)$-Rickart in $\mathcal{A}$ and for every morphism $f:M\to N$, ${\rm Ker}(f)$ is $M$-cyclic.
\item $R(M)$ is dual strongly $R(N)$-Rickart in $\mathcal{A}$ and for every morphism $f:M\to N$, ${\rm Ker}(f)\in {\rm Refl}(R)$.
\end{enumerate}
\item Assume that $(L,R)$ is right adjoint. Let $M$ and $N$ be objects of $\mathcal{A}$ such that $M,N\in {\rm
Refl}(L)$. Then the following are equivalent:
\begin{enumerate}[(i)] 
\item $N$ is dual strongly $M$-Rickart in $\mathcal{A}$.
\item $L(M)$ is strongly $L(N)$-Rickart in $\mathcal{B}$ and for every morphism $f:M\to N$, ${\rm Coker}(f)$ is $N$-cocyclic.
\item $L(M)$ is strongly $L(N)$-Rickart in $\mathcal{B}$ and for every morphism $f:M\to N$, ${\rm Coker}(f)\in {\rm Refl}(L)$.
\end{enumerate}
\end{enumerate}
\end{theo}

Next we discuss the transfer of (dual) strong relative Rickart property to endomorphism rings of (graded) modules. 

Let us recall some terminology and notation in module categories. We need the concepts of locally
split monomorphism and locally split epimorphism due to Azumaya. Recall that a monomorphism $f:A\to B$ of right
$R$-modules is called \emph{locally split} if for every $a\in A$, there exists an $R$-homomorphism $h:B\to A$ such that
$h(f(a))=a$, while an epimorphism $g:B\to C$ of right $R$-modules is called \emph{locally split} if for every $c\in C$,
there exists an $R$-homomorphism $h:C\to B$ such that $g(h(c))=c$ \cite[p.~132]{Azumaya}. 

Recall that a right $R$-module $M$ is called \emph{quasi-retractable} if for every family $(f_i)_{i\in I}$ with each
$f_i\in {\rm End}_R(M)$ and $\bigcap_{i\in I} {\rm Ker}(f_i)\neq 0$, ${\rm Hom}_R(M,\bigcap_{i\in I} {\rm Ker}(f_i))\neq
0$ \cite[Definition~2.3]{RR09}. Dually, a right $R$-module is called \emph{quasi-coretractable} if for every family
$(f_i)_{i\in I}$ with each $f_i\in {\rm End}_R(M)$ and $\sum_{i\in I} {\rm Im}(f_i)\neq M$, ${\rm Hom}_R(M/\sum_{i\in I}
{\rm Im}(f_i),M)\neq 0$ \cite[Definition~3.2]{KST}. 

We also need the following notions.

\begin{defn} \cite[Definition~4.1]{CO} \rm A right $R$-module $M$ is called:
\begin{enumerate}
\item \emph{$k$-quasi-retractable} if ${\rm Hom}_R(M,{\rm Ker}(f))\neq 0$ for every $f\in {\rm End}_R(M)$
with ${\rm Ker}(f)\neq 0$.
\item \emph{$c$-quasi-coretractable} if ${\rm Hom}_R({\rm Coker}(f),M)\neq 0$ for every $f\in {\rm End}_R(M)$ with ${\rm
Coker}(f)\neq 0$.
\end{enumerate}
\end{defn}

The following theorem gives several equivalent conditions relating the (dual) strong self-Rickart properties of a module and
of its endomorphism ring. 

\begin{theo} \label{t1:end} Let $M$ be a right $R$-module, and let $S={\rm End}_R(M)$. 
\begin{enumerate}
\item The following are equivalent:
\begin{enumerate}[(i)] 
\item $M$ is a strongly self-Rickart right $R$-module. 
\item $S$ is a strongly self-Rickart right $S$-module and for every $f\in S$, ${\rm Ker}(f)$ is $M$-cyclic.
\item $S$ is a strongly self-Rickart right $S$-module and for every $f\in S$, ${\rm Ker}(f)\in {\rm Stat}({\rm Hom}_R(M,-))$.
\item $S$ is a strongly self-Rickart right $S$-module and for every $f\in S$, ${\rm ker}(f)$ is a locally split monomorphism.
\item $S$ is a strongly self-Rickart right $S$-module and $M$ is $k$-quasi-retractable.
\end{enumerate}
\item The following are equivalent:
\begin{enumerate}[(i)] 
\item $M$ is a dual strongly self-Rickart right $R$-module.
\item $S$ is a strongly self-Rickart left $S$-module and for every $f\in S$, ${\rm Coker}(f)$ is $M$-cocyclic.
\item $S$ is a strongly self-Rickart left $S$-module and for every $f\in S$, ${\rm Coker}(f)\in {\rm Refl}({\rm Hom}_R(-,M))$.
\item $S$ is a strongly self-Rickart left $S$-module and for every $f\in S$, ${\rm coker}(f)$ is a locally split epimorphism.
\item $S$ is a strongly self-Rickart left $S$-module and $M$ is $c$-quasi-coretractable.
\end{enumerate}
\end{enumerate}
\end{theo}

\begin{proof} This follows by Theorems \ref{t1:equiv} and \ref{t1:dual}, Corollary \ref{c1:wduo} and \cite[Theorem~4.3]{CO}. 
\end{proof}

For a $G$-graded ring $R$ and two graded right $R$-modules $M$ and $N$ 
we may consider the graded abelian group $\HOM{R}{M}{N}$, whose
$\sigma$-th homogeneous component is $$\HOM{R}{M}{N}_{\sigma} = \{f\in \hom{R}{M}{N} \,\mid\, f(M_{\lambda})\subseteq
M_{\sigma\lambda} \mbox{ for all }\sigma\in G \}.$$ For $M=N$, $S={\rm END}_R(M)=\HOM{R}{M}{M}$ is a $G$-graded ring
and $M$ is a graded $(S,R)$-bimodule, that is, $S_{\tau}\cdot M_{\sigma}\cdot R_{\lambda}\subseteq
M_{\tau\sigma\lambda}$ for every $\tau,\sigma,\lambda\in G$. 

For a graded right $S$-module $N$, the right $R$-module $N\otimes_SM$ may be graded by
$$(N\otimes_SM)_{\tau}=\left \{\sum_{\sigma\lambda=\tau}n_{\sigma}\otimes m_{\lambda} \mid n_{\sigma}\in
N_{\sigma}, m_{\lambda}\in M_{\lambda}\right \}.$$ Then the covariant functor $-\otimes_SM:{\rm gr}(S)\to {\rm gr}(R)$
is a left adjoint to the covariant functor $\HOM{R}{M}{-}:{\rm gr}(R)\to {\rm gr}(S)$.

For a graded right $R$-module $M$ with $S={\rm END}_R(M)$, the functors $\HOM{R}{-}{M}:{\rm gr}(R)\to {\rm gr}(S^{\rm
op})$ and $\HOM{S}{-}{M}:{\rm gr}(S^{\rm op})\to {\rm gr}(R)$ form a right adjoint pair of contravariant functors.

\begin{coll} \label{c1:endgr} Let $M$ be a graded right $R$-module, and let $S={\rm END}_R(M)$. 
\begin{enumerate} \item The following are equivalent:
\begin{enumerate}[(i)]
\item $M$ is a strongly self-Rickart graded right $R$-module.
\item $S$ is a strongly self-Rickart graded right $S$-module and for every $f\in S$, ${\rm Ker}(f)$ is $M$-cyclic.
\item $S$ is a strongly self-Rickart graded right $S$-module and for every $f\in S$, ${\rm Ker}(f)\in {\rm Stat}(\HOM{R}{M}{-})$.
\end{enumerate}
\item The following are equivalent:
\begin{enumerate}[(i)]
\item $M$ is a dual strongly self-Rickart graded right $R$-module.
\item $S$ is a strongly self-Rickart graded left $S$-module and for every $f\in S$, ${\rm Coker}(f)$ is $M$-cocyclic.
\item $S$ is a strongly self-Rickart graded left $S$-module and for every $f\in S$, ${\rm Coker}(f)\in {\rm
Refl}(\HOM{R}{-}{M})$.
\end{enumerate}
\end{enumerate}
\end{coll}

\begin{proof} (1) Consider the adjoint pair of covariant functors $(T,H)$, where $$T=-\otimes_SM:{\rm gr}(S)\to {\rm
gr}(R),$$ $$H=\HOM{R}{M}{-}:{\rm gr}(R)\to {\rm gr}(S).$$ Let $\varepsilon:TH\to 1_{{\rm gr}(R)}$ be the counit of
adjunction. We have $TH(M)\cong M$, hence $\varepsilon_M$ is an isomorphism, and so $M\in {\rm Stat}(H)$. Now take $M=N$
in Theorem \ref{t1:equiv} (1). 

(2) Consider the right adjoint pair of contravariant functors $(H,H')$, where $$H=\HOM{R}{-}{M}:{\rm gr}(R)\to
{\rm gr}(S^{\rm op}),$$ $$H'=\HOM{S}{-}{M}:{\rm gr}(S^{\rm op})\to {\rm gr}(R).$$ Let $\eta:1_{{\rm gr}(S^{\rm
op})}\to H'H$ be the unit of adjunction. We have $H'H(M)\cong M$, hence $\eta_M$ is an isomorphism, and so $M\in {\rm
Refl}(H)$. Now take $M=N$ in Theorem \ref{t1:dual} (2). 
\end{proof}

\end{document}